\documentclass{amsart}
\usepackage{graphicx}
\newtheorem{thm}{Theorem}[section]

\newtheorem{lem}[thm]{Lemma}
\newtheorem{prop}[thm]{Proposition}
\theoremstyle{definition}
\newtheorem{defn}[thm]{Definition}
\theoremstyle{remark}

\numberwithin{equation}{section}
\newcommand{\norm}[1]{\left\Vert#1\right\Vert}

\newcommand{\set}[1]{\left\{#1\right\}}

\begin{document}

\title[Upper and lower estimates of norms in variable exponent spaces]{On the upper and lower estimates of norms in variable exponent spaces}%
\author{Tengiz Kopaliani}%
\author{Nino Samashvili} %
\author{Shalva Zviadadze}
\address[Tengiz Kopaliani]{Faculty of Exact and Natural Sciences, Javakhishvili Tbilisi State University, 13, University St., Tbilisi, 0143, Georgia}%
\address[Nino Samashvili]{Faculty of Exact and Natural Sciences, Javakhishvili Tbilisi State University, 13, University St., Tbilisi, 0143, Georgia}%
\address[Shalva Zviadadze]{Faculty of Exact and Natural Sciences, Javakhishvili Tbilisi State University, 13, University St., Tbilisi, 0143, Georgia}%
\email[Tengiz Kopaliani]{tengizkopaliani@gmail.com}%
\email[Nino Samashvili]{n.samashvili@gmail.com}
\email[Shalva Zviadadze]{sh.zviadadze@gmail.com}

\thanks{The research of the first two authors is supported by Shota Rustaveli National Science Foundation grant \#DI/9/5-100/13. Research of last author supported by Shota Rustaveli National Science Foundation grant \#52/36.}



\begin{abstract}
In the present paper we investigate some geometrical properties of the norms in Banach function spaces. Particularly there is shown that if exponent $1/p(\cdot)$ belongs to $BLO^{1/\log}$ then for the norm of corresponding variable exponent Lebesgue space we have the following lower estimate
$$\left\|\sum \chi_{Q}\|f\chi_{Q}\|_{p(\cdot)}/\|\chi_{Q}\|_{p(\cdot)}\right\|_{p(\cdot)}\leq
C\|f\|_{p(\cdot)}$$
where $\{Q\}$ defines disjoint partition of $[0;1]$. Also we have constructed variable exponent
Lebesgue space with above property which does not possess following upper estimation
$$\|f\|_{p(\cdot)}\leq C\left\|\sum \chi_{Q}\|f\chi_{Q}\|_{p(\cdot)}/\|\chi_{Q}\|_{p(\cdot)}\right\|_{p(\cdot)}.
$$

\end{abstract}
\maketitle
\section{Introduction}

Let $\Omega\subset\mathbb{R}^{n}$ and let $\mathcal{M}$
 be the space of all equivalence classes of Lebesgue
 measurable real-valued functions endowed
 with the topology of convergence in measure relative to each
 set of finite measure.

\begin{defn}
\label{defn_BFS}
 A Banach  subspace $X$ of $\mathcal{M}$ is called a Banach function
 space (BFS) on $\Omega$ if

$1)$ the norm $\|f\|_{X}$ is defined for every measurable
function $f$ and $f\in X$ if and only if
$\|f\|_{X}<\infty$. $\|f\|_{X}=0$ if and only if $f=0$ a.e.;

$2)$ $\||f|\|_{X}=\|f\|_{X}$ for all $f\in X$;

$3)$ if $0\leq f\leq g$ a.e., then $\|f\|_{X}\leq\|g\|_{X}$;

$4)$ if $0\leq f_{n}\uparrow f$ a.e., then $\|f_{n}\|_{X}\uparrow
\|f\|_{X}$;

$5)$ if $E$ is measurable subset of $\Omega$ such that $|E|<
\infty,$ (below we denote the Lebesgue measure of $E$ by $|E|$) then
$\|\chi_{E}\|_{X}<\infty$;

$6)$ for every measurable set $E,\,|E|<\infty,$ there is a constant
$C_{E}<\infty$ such that $\int_{E}f(t)dt\leq C_{E}\|f\|_{X}.$
\end{defn}

\par Given a BFS  $X,$  its associate space $X'$ is defined by
$$
X'=\left\{g:\,\,\,\,\int_{\Omega}|f(x)g(x)|dx<\infty\,\,\,\mbox{for
all}\,\,f\in X\right\}
$$
and endowed with the associate norm
$$
\|f\|_{X'}=\sup\left\{\int_{\Omega}|f(x)g(x)|dx\,\,:\,
\,\,\|g\|_{ X}\leq1\right\}.
$$
An immediate consequence of this definition is the generalized
H\"{o}lder's  inequality: for all $f\in X$ and $g\in X',$
$$
\left|\int_{\Omega}f(x)g(x)dx\right|\leq\|f\|_{X}\|g\|_{X'}.
$$
Furthermore, $X'$ is also  a BFS on $\Omega$ and $(X')'=X.$  The
associate space of $X$ is  closed norming subspace of the dual space
$X^{\ast}$, and equality
$$
\|f\|_{X}=\sup\left\{\int_{\Omega}|f(x)g(x)|dx\,\,:\,
\,\,\|g\|_{ X'}\leq1\right\}
$$
holds for all $f\in X$ (see \cite{BS}).

Given a Banach function space $X,$ define the scale of spaces
$X^{r},\,\,0<r<\infty,$ by
$$
X^{r}=\{f\in \mathcal{M}:\,\,|f|^{r}\in X\},
$$
with the "norm"
$$
\|f\|_{X^{r}}=\||f|^{r}\|^{1/r}_X.
$$
If $r\geq1,$ then $\|\cdot\|_{X^{r}}$ is again an actual norm and
$X^{r}$ is a Banach function space. However, if $r<1,$ need not be a
Banach function space. The simple example is the scale of
Lebesgue spaces: if $X=L^{p}(\Omega),\,(1\leq p<\infty),$
then $(L^{p})^{r}=L^{pr},$ and so $X^{r}$ is a Banach space only for
$r\geq1/p.$

 Let $\Im$ be some  fixed  family of   sequences
$\mathcal{Q}=\{Q_{i}\}$ of disjoint measurable  subsets of
$\Omega,\,|Q_{i}|>0$ such that
$\Omega=\cup_{Q_{i}\in\mathcal{Q}}Q_{i}.$ We ignore the difference
in notation caused by a null set.

Everywhere in the sequel $l_{\mathcal{Q}}$ is a Banach sequential
space (BSS), meaning that axioms 1)-6) from definition
\ref{defn_BFS} are satisfied with respect to the count measure. Let $e_{k}=e_{Q_{k}}$ denote the standard unit vectors in
$l_{\mathcal{Q}}.$

\par Kopaliani  in \cite{K1} introduced notions of uniformly upper (lower)
$l$-estimates.

\begin{defn}
 1)  Let $l=\left\{l_{\mathcal{Q}}\right\}_{\mathcal{Q}\in\Im}$ be a
family of BSSs. A BFS $X$ is said to satisfy a uniformly upper
$l$-estimate if there exists a constant $C<\infty$ such that for
every $f\in X$ and $\mathcal{Q}\in \Im$ we have
$$\|f\|_{X}\leq C\left\|\sum_{Q_{i}\in\mathcal{Q}}e_{i}\|f\chi_{Q_{i}}\|_{X}\right\|_{l_{\mathcal{Q}}}.$$
2) BFS $X$ is said to satisfy uniformly lower $l$-estimate if there
exists a constant $C<\infty$ such that for every $f\in X$ and
$\mathcal{Q}\in\Im$ we have
$$\|f\|_{X}\geq C\left\|\sum_{Q_{i}\in\mathcal{Q}}e_{i}\|f\chi_{Q_{i}}\|_{X}\right\|_{l_{\mathcal{Q}}}.$$
\end{defn}

Note that if in Definition 1.2 for
 all $\mathcal{Q}\in\Im,$  we take   one  discrete  Lebesgue
space $l_{p},\,(1\leq p<\infty),$  we obtain classical definition of
upper and lower $p$-estimates of Banach spaces
(see \cite{Sh}, \cite{FJ}). The existence of upper or lower $p$-estimates  in the  Banach spaces  is of great
interest in study of the structure of the space (see \cite{LZ}).
 Berezhnoi \cite{Be1, Be2} investigate uniformly upper (lower)
$l$-estimates of BFS, when discrete $l_{\mathcal{Q}}$ spaces for all
partition of $\Omega$ coincides to some discrete BSS.

\begin{defn}
\label{G_prop} A pair of BFSs $(X,Y)$ is said to have property $G$
if there exists a constant $C>0$ such that
$$\sum_{Q_{i}\in\mathcal{Q}}\|f\chi_{Q_{i}}\|_{X}\cdot\|g\chi_{Q_{i}}\|_{Y'}\leq C \cdot\|f\|_{X}\cdot\|g\|_{Y'}$$
for any $\mathcal{Q}\in \Im$ and every $f\in X,\,g\in Y'.$
\end{defn}

Definition \ref{G_prop} was introduced by Berezhnoi \cite{Be2}. Let
us remark that a pair $(L^{p}(\Omega),L^{q}(\Omega))$ possesses the
property $G$ if $p\leq q.$

The  connections between the property $G$ and uniformly upper (lower)
$l$-estimates of BFS-s was investigated in paper \cite{K1}.

\begin{thm}[\cite{K1}]
\label{theorem_equval_G_l_estimates}
 Let $(X,Y)$ be a pair of BFSs. Then the following assertions
are equivalent:

$1)$ The pair $(X,Y)$ of BFSs possesses property $G$.

$2)$ There is a family
$l=\left\{l_{\mathcal{Q}}\right\}_{\mathcal{Q}\in \Im}$ of BSSs such
that $X$ satisfies uniformly lower $l$-estimate and $Y$ satisfies
uniformly upper $l$-estimate.
\end{thm}

\begin{thm}[\cite{K1}]
\label{theorem_G_imply_G'_G''}
Let the pair $(X,X)$ of BFSs possesses property $G.$ Then there
exist constants $C_{1},C_{2}>0$ such that for every $f\in X$ and
$\mathcal{Q}\in \Im$ we have
\begin{equation}
\label{estim_upper_and_lower}
C_{1}\|f\|_{X}\leq\left\|\sum_{Q\in\mathcal{Q}}\frac{\|f\chi_{Q}\|_{X}}{\|\chi_{Q}\|_{X}}\chi_{Q}\right\|_{X}\leq
C_{2}\|f\|_{X}.
\end{equation}
\end{thm}

\par Note that the (\ref{estim_upper_and_lower}) type inequalities is very important for
studying the boundedness properties of operators of harmonic analysis
in variable Lebesgue spaces (see \cite{CUF}, \cite{DHHR}).

\begin{defn}
\label{G'_G''_prop} We say that  BFS $X$ has property $G'_{}$
(property $G''$) if there exists constant $C_{1}$ ($C_{2}>0$) such
that for every $f\in X$ and $\mathcal{Q}\in \Im$ we have
\begin{equation}
\label{defn_prop_G'_G''}
\left\|\sum_{Q\in\mathcal{Q}}\frac{\|f\chi_{Q}\|_{X}}{\|\chi_{Q}\|_{X}}\chi_{Q}\right\|_{X}\leq
C_{1}\|f\|_{X},\,\,\,\,\left(\|f\|_{X}\leq
C_{2}\left\|\sum_{Q\in\mathcal{Q}}\frac{\|f\chi_{Q}\|_{X}}{\|\chi_{Q}\|_{X}}\chi_{Q}\right\|_{X}\right).
\end{equation}
\end{defn}

 The idea of (\ref{defn_prop_G'_G''}) type inequalities are to generalize the following property of the Lebesgue norm
$$
\|f\|_{L^{p}}=\left\|\sum_{i}\frac{\|f\chi_{\Omega_{i}}\|_{L^{p}}}{\|\chi_{\Omega_{i}}\|_{L^{p}}}\chi_{\Omega_{i}}\right\|_{L^{p}},
$$
where $\Omega_i$ is disjoint measurable partition of $\Omega$.

\par The aim of our paper is to investigate the properties $G'$ (property $G''$)
for variable Lebesgue spaces $L^{p(\cdot)}[0;1].$  By $\Im$ we
denote the family of all sequences (may be finite) $\{Q_{i}\}$ of
disjoint intervals from $[0;1]$. Assume that sets like $[0;a)$
and $(b;1]$ are also intervals. We have described the class  of exponents, for which the correspondent variable exponent Lebesgue spaces has
property $G'$ (property $G''$). Also we have constructed variable exponent
Lebesgue space with property $G'$ ($G''$), which does not possess $G''$ ($G'$) property.

\par Particularly we will proof following theorems:
\begin{thm}
\label{theorem_BLO_G'} Let for exponent $p(\cdot)$ we have
$1/p(\cdot)\in BLO^{1/\log},\,1\leq p_-\leq p_+<\infty$. Then the
space $L^{p(\cdot)}[0;1]$ has property $G'$.
\end{thm}

\begin{thm}
\label{theorem_BLO+C+G''} Let for exponent $p(\cdot),1\leq p_-\leq
p_+<\infty$ we have $1/p(\cdot)\in BLO^{1/\log}$. Then there exists
$c$ such that the space $L^{(p(\cdot)+c)'}[0;1]$ has property
$G''$.
\end{thm}

\begin{thm}
\label{theorem_BLO_G'_notG''} \textsc{1)} There exists exponent $p(\cdot),1\leq
p_-\leq p_+<\infty$ such that $1/p(\cdot)\in BLO^{1/\log}$ and
$L^{p(\cdot)}[0;1]$ has property $G'$ but does not have property $G''$.
\\\textsc{2)} There exists exponent $p(\cdot),1\leq
p_-\leq p_+<\infty$ such that $1/p(\cdot)\in BLO^{1/\log}$ and
$L^{p(\cdot)}[0;1]$ has property $G''$ but does not have property $G'$.
\end{thm}

\section{Some remarks on properties $G'$ and $G''$}

In this section we will discuss about relations between $G'$ and $G''$ properties for BFS $X$ and its associate space. $\Im$ denotes the family of all sequences of disjoint intervals. 
\begin{defn}[\cite{Be2}]
Let $X$ be a BFS. We say that for BFS $X$ is fulfilled condition $A$
if there exists constant $C>0$ such that, for all interval
$Q\subset[0;1]$
$$
\|\chi_{Q}\|_{X}\cdot\|\chi_{Q}\|_{X'}\leq C\cdot |Q|
$$
\end{defn}

\begin{thm}
\label{theorem_G'_and_condition_A_imply_G''}
Let BFS $X$ has property $G'$ and for $X$ fulfilled condition $A.$
Then associate space of $X$ has property $G''.$
\end{thm}

\begin{proof}
Let $\mathcal{Q}=\{Q_1,Q_2,...\}$ denotes some partition of
$[0;1].$  Let $g\in X'$ and $f\in X$ such that $\|f\|_{X}\leq1.$
Using H\"{o}lders inequality and $A$ condition we conclude that
$|Q|\asymp\|\chi_{Q}\|_{X}\cdot\|\chi_{Q}\|_{X'}.$ Using this fact and
property $G'$ we obtain
$$
\int_{[0;1]}|f(x)g(x)|dx=\sum_{k}\int_{Q_{k}}|f(x)g(x)|dx\leq\sum_{k}\|f\chi_{Q}\|_{X}\cdot\|g\chi_{Q}\|_{X'}
$$
$$
\leq C_1\int\limits_{[0;1]}\sum_{k}\frac{\|f\chi_{Q}\|_{X}}{\|\chi_{Q}\|_{X}}\frac{\|g\chi_{Q}\|_{X'}}{\|\chi_{Q}\|_{X'}}\chi_{Q_{k}}dx
\leq C_1\left\|\sum_{k}\frac{\|f\chi_{Q}\|_{X}}{\|\chi_{Q}\|_{X}}\chi_{Q_{k}}\right\|_{X}
\left\|\sum_{k}\frac{\|g\chi_{Q}\|_{X'}}{\|\chi_{Q}\|_{X'}}\chi_{Q_{k}}\right\|_{X'}
$$
$$
\leq
C_2\left\|\sum_{k}\frac{\|g\chi_{Q}\|_{X'}}{\|\chi_{Q}\|_{X'}}\chi_{Q_{k}}\right\|_{X'}.
$$

Consequently $X'$ possess $G''$ property. 
\end{proof}

\par Note that if for BFS $X$ we have property $G$ then for $X'$ we have property $G$ without condition $A$ (see \cite{K1}).

\begin{defn}
Let $\mathcal{Q}\in\Im$. We define averaging operator with respect to $\mathcal{Q}$ by
$$
T_{\mathcal{Q}}f(x)=\sum_i |f|_{Q_i}\chi_{Q_i}(x)
$$
\end{defn}
where $|f|_{Q}$ denotes the average of $|f|$ on $Q$.
\begin{thm}
Let BFS $X$ has property $G''$ and averaging operators $T_{\mathcal{Q}}:X\to X$, $\mathcal{Q}\in\Im$ are uniformly bounded. Then associate space of $X$ has property $G'.$
\end{thm}

\begin{proof}
Let $g\in X$ is nonnegative function such that $||g||_{X}\leq1$. For any $\varepsilon>0$ and $i$ we choose nonnegative function $h_i\in X$ such that $||h_i||_X\leq1$ and $||f\chi_{Q_i}||_{X'}\leq(1+\varepsilon)\int_{Q_i}fh_i$. Note that uniformly boundedness of averaging operator implies condition $A$ for space $X$ (see \cite{Be2}). So by property $G''$ and H\"{o}lder inequality we get
$$
\int\limits_{[0;1]}g(x)\sum_i\frac{||f\chi_{Q_i}||_{X'}}{||\chi_{Q_i}||_{X'}}\chi_{Q_i}(x)dx\leq\int\limits_{[0;1]}g(x)\sum_i\frac{(1+\varepsilon)\int_{Q_i}f(t)h_i(t)dt}{||\chi_{Q_i}||_{X'}}\chi_{Q_i}(x)dx
$$
$$
=(1+\varepsilon)\int_{[0;1]}f(t)\sum_i\frac{h_i(t)\int_{Q_i}g(x)dx}{||\chi_{Q_i}||_{X'}}dt\leq
(1+\varepsilon)||f||_{X'}\left\|\sum_i\frac{h_i(\cdot)\int_{Q_i}g(x)dx}{||\chi_{Q_i}||_{X'}}\chi_{Q_i}\right\|_{X}
$$
$$
\leq(1+\varepsilon)||f||_{X'}\left\|\sum_i\frac{\|h_i\|_{X}\int_{Q_i}g(x)dx}{||\chi_{Q_i}||_{X}||\chi_{Q_i}||_{X'}}\chi_{Q_i}\right\|_X\leq C_1(1+\varepsilon)||f||_{X'}\left\|\sum_i\frac{\chi_{Q_i}}{|Q_i|}\int_{Q_i}g(x)dx\right\|_X
$$
$$
\leq C_2(1+\varepsilon)||f||_{X'}||g||_{X}\leq C_2(1+\varepsilon)||f||_{X'}.
$$
By the fact that $\varepsilon$ is arbitrary small we conclude that $X'$ has property $G'$.
\end{proof}

Note that if $0<r<\infty$ then for any $f\in X$ we have
$\|f\|_{X}=\|f^{1/r}\|_{X^{r}}^{r}$ and the inequalities in definition \ref{G'_G''_prop} can be written in following form
$$
\left\|\sum_{Q\in\mathcal{Q}}\frac{\|f^{1/r}\chi_{Q}\|_{X^{r}}}{\|\chi_{Q}\|_{X^{r}}}\chi_{Q}\right\|_{X^{r}}^{r}\leq
C_{1}\|f^{1/r}\|_{X^{r}}^{r},
$$
$$
\|f^{1/r}\|_{X^{r}}^{r}\leq
C_{2}\left\|\sum_{Q\in\mathcal{Q}}\frac{\|f^{1/r}\chi_{Q}\|_{X^{r}}}{\|\chi_{Q}\|_{X^{r}}}\chi_{Q}\right\|_{X^{r}}^{r}.
$$
Consequently if BFS has property $G'$ ($G''$), then the
"norms" $\|\cdot\|_{X^{r}}\, (0<r<\infty)$ have property
$G'$ ($G''$).

\section{Variable Lebesgue spaces}

The variable exponent Lebesgue spaces $L^{p(\cdot)}(\mathbb{R}^{n})$
and the corresponding variable exponent Sobolev spaces
$W^{k,p(\cdot)}$ are of interest for their applications to the
problems in  fluid dynamics,  partial differential equations with
non-standard growth conditions, calculus of variations,  image
processing and etc (see \cite{DHHR}).

Given a measurable function
$p:[0;1]\rightarrow[1;+\infty),\,\,L^{p(\cdot)}[0;1]$ denotes
the set of measurable functions $f$ on $[0;1]$ such that for some
$\lambda>0$
$$
\int_{[0;1]}\left(\frac{|f(x)|}{\lambda}\right)^{p(x)}dx<\infty.
$$
This set becomes a Banach function spaces when equipped with the
norm
$$
\|f\|_{p(\cdot)}=\inf\left\{\lambda>0:\,\,
\int_{[0;1]}\left(\frac{|f(x)|}{\lambda}\right)^{p(x)}dx\leq1\right\}.
$$
\par For the given $p(\cdot),$ the conjugate exponent $p'(\cdot)$ is defined pointwise $p'(x)=p(x)/
(p(x)-1)$, $x\in[0;1].$ Given a set $Q\subset[0;1]$ we define
 some standard notations:
$$
p_-(Q):=\mathop{\mbox{essinf}}\limits_{x\in Q}p(x),\:\:\:\:
p_+(Q):=\mathop{\mbox{esssup}}\limits_{x\in Q}
p(x),\:\:\:\:p_-:=p_-([0;1]),\:\:\:\: p_+:=p_+([0;1]).
$$
In the notation introduced above, an exponent
$p(\cdot),\,\,1\leq p_{-}\leq p_{+}<\infty,$ the associate space of
$L^{p(\cdot)}[0;1]$ contains measurable functions $f$ such that
$$
\|f\|_{(L^{p(\cdot)})'}'=\sup\left\{\int_{[0;1]}|f(x)g(x)|dx\,:\,g\in
L^{p(\cdot)}[0;1],\,\|g\|_{p(\cdot)}\leq1\right\}<\infty.
$$
Note that in this case the associate space of $L^{p(\cdot)}[0;1]$ is
equal to $L^{p'(\cdot)}[0;1],$ and $\|\cdot\|_{(L^{p(\cdot)})'}'$ and $\|\cdot\|_{p'(\cdot)}$ are equivalent norms (see \cite{CUF}, \cite{DHHR}). We have also
$$
\int_{[0;1]}|f(x)g(x)|dx\leq
C\|f\|_{p(\cdot)}\|g\|_{p'(\cdot)},\,\,\,f\in L^{p(\cdot)}[0;1],\,\,g\in
L^{p'(\cdot)}[0;1].
$$
Conversely for all $f\in L^{p(\cdot)}[0;1]$
$$
\|f\|_{p(\cdot)}\leq C\sup\int_{[0;1]}|f(x)g(x)|dx,
$$
where the supremum is taken over all $g\in L^{p'(\cdot)}[0;1]$ such
that $\|g\|_{p'(\cdot)}\leq 1.$

 Given exponent $p(\cdot),\,1\leq p_{-}\leq p_{+}<\infty$
and a Lebesgue measurable function $f$ define the modular associated
with $p(\cdot)$ on the set $E\subset[0;1]$ by
$$
\rho_{p(\cdot),E}f=\int_{E}|f(x)|^{p(x)}dx.
$$
In case of constant exponents, the $L^{p}$ norm and the modular
differ only by an exponent. In the variable Lebesgue spaces their
relationship is more subtle as the next result shows
(see \cite{CUF}, \cite{DHHR}).
\begin{prop}
\label{proposition_modular_in_var}
Given exponent $p(\cdot),$ suppose $1\leq p_{-}\leq p_{+}<\infty.$
Let $E$ measurable subset of $[0;1].$ Then:

$(1)$ $\|f\chi_{E}\|_{p(\cdot)}=1$ if and only if
$\rho_{p(\cdot),E}f=1$;

$(2)$ if $\rho_{p(\cdot),E}f\leq C,$ then
$\|f\chi_{E}\|_{p(\cdot)}\leq \max(C^{1/p_-(E)},\,C^{1/p_+(E)});$

$(3)$ if $\|f\|_{p(\cdot)}\leq C,$ then
$\rho_{p(\cdot),E}f\leq\max(C^{p_+(E)},C^{p_-(E)}).$
\end{prop}

The next result is a necessary and sufficient condition for the
embedding $L^{q(\cdot)}[0;1]$ $\subset L^{p(\cdot)}[0;1]$
(see \cite{CUF}, \cite{DHHR}).

\begin{prop}
\label{proposition_embeding_Lp_Lq}
Given the exponents $p(\cdot), q(\cdot),$  suppose $1\leq p_{-}\leq
p_{+}<\infty, 1\leq q_{-}\leq q_{+}<\infty.$ Then
$L^{q(\cdot)}[0;1]\subset L^{p(\cdot)}[0;1]$ if and only if
$p(\cdot)\leq q(\cdot)$ almost everywhere. Furthermore, in this case
we have
$$
\|f\|_{p(\cdot)}\leq 2\|f\|_{q(\cdot)}.
$$
\end{prop}

For our results we need to impose some regularity on the exponent
function  $p(\cdot).$ The most important  condition, one widely used
in the study of variable Lebesgue spaces, is log-L\"{o}lder
continuity. Let $C^{1/\log}$ denotes the set of exponents
$p:[0;1]\rightarrow[1,+\infty)$ with log-H\"{o}lder condition
\begin{equation}
|(p(x)-p(y))\ln|x-y||\leq C,\,\, x,\,y\in[0;1],\,x\neq y.
\end{equation}

\par For Lebesgue integrable function $f$ define Hardy-Littlewood maximal function
$$
Mf(x)=\sup\limits_{x\in Q}|f|_Q,
$$
where supremum is taken over all $Q\subset [0;1]$ intervals
containing point $x$  and $f_{Q}$ denotes the average of function $f$
on $Q.$  Let by $\mathcal{B}$ denote set of all exponents $p(\cdot)$ for which
Hardy-Littlewood maximal operator is bounded on the space
$L^{p(\cdot)}[0;1]$. Diening \cite{DI} proved  a key consequence
of log-H\"{o}lder continuity of $p(\cdot)$. If $1<p_-$ and
$p(\cdot)\in C^{1/\log},$ then $p(\cdot)\in \mathcal{B}.$

\par Kopaliani \cite{K1} proved that if exponent $p(\cdot)$ satisfies log-H\"{o}lder conditions then
the pair of BFSs $(L^{p(\cdot)}[0;1],L^{p(\cdot)}[0;1])$
 has property $G$. Note that there are another classes of exponents
 $p(\cdot)$ such that pair of BFSs $(L^{p(\cdot)}[0;1],L^{p(\cdot)}[0;1])$ has property
 $G.$ For instance, if  exponent $p(\cdot)$ is absolutely continuous on
 $[0;1]$, then the pair of BFSs  $(L^{p(\cdot)}[0;1],L^{p(\cdot)}[0;1])$ has property $G$ (see
\cite{K3}).  Note also that there exists continuous exponent on
$[0;1]$ such that the pair of BFSs $(L^{p(\cdot)}[0;1],
L^{p(\cdot)}[0;1])$ does not have property $G$ (see \cite{K1}).

Given a function $f\in L^{1}[0;1]$. Let define its $BMO$ modulus by
$$
\gamma(f,r)=\sup\limits_{|Q|\leq
r}\frac{1}{|Q|}\int_{Q}|f(x)-f_{Q}|dx,\,\,0<r\leq1,
$$
where supremum is taken over all intervals of $[0;1].$ We say that $
f\in BMO^{1/\log}$ if $\gamma(f,r)\leq C/\log(e+1/r)$ and $ f\in
VMO^{1/\log}$ if $\gamma(f,r)\log(e+1/r)\rightarrow 0$ as
$r\rightarrow0.$

Given a function $f\in L^{1}[0;1]$. Let define its $BLO$ modulus by
$$
\eta(f,r)=\sup\limits_{|Q|\leq
r}(f_{Q}-\mathop{\mbox{essinf}}\limits_{x\in Q}f(x)),\,\,0<r\leq1,
$$
where supremum is taken over all intervals of $[0;1].$ We say that $
f\in BLO^{1/\log}$ if $\eta(f,r)\leq C/\log(e+1/r)$.

\par The class $BMO^{1/\log}$ is very important for investigation of
exponents from $\mathcal{B}.$
\begin{thm} [\cite{Le}, \cite{KK}]
\label{theorem_BMO_BLO_VLO}
Let $p:[0;1]\to[1;+\infty)$, then\\
\textsc{1)} if $p(\cdot)\in \mathcal{B}$, then $1/p(\cdot)\in BMO^{1/\log}$;\\
\textsc{2)} if $p(\cdot)\in VMO^{1/\log}$, then $p(\cdot)\in \mathcal{B}$;\\
\textsc{3)} if $p(\cdot)\in BMO^{1/\log}$, then there exists $c$ such that
$p(\cdot)+c\in \mathcal{B}$.
\end{thm}

\section{Proof of results}

\begin{proof}[Proof of theorem \ref{theorem_BLO_G'}]

We begin with auxiliary estimations.
\begin{lem}
Let $p(\cdot)$  be a exponent on $[0;1]$ with $1\leq p_{-}\leq
p_{+}<\infty.$ Then for all $t\geq0$ and $Q\subset[0;1]$ interval
\begin{equation}
\label{estim_mean_of_t_p}
\frac{1}{|Q|}\int_{Q}t^{p(x)}dx\geq
e^{2(p_{-}(Q)-p_{+}(Q))}t^{\overline{p}_{Q}},
\end{equation}
where $\overline{p}_{Q}$ is defined as
$\frac{1}{\overline{p}_{Q}}=\frac{1}{|Q|}\int_{Q}\frac{1}{p(x)}dx$.
\end{lem}

This lemma is proved in  \cite{Di} (see Lemma 4.1) in case of $1<
p_{-}\leq p_{+}<\infty,$ but analogously may be proved in presented
case. If in (\ref{estim_mean_of_t_p}) we take $t=\frac{1}{\|\chi_{Q}\|_{p(\cdot)}},$
we obtain
\begin{equation}
\label{estim_chi_norm_lower}
\|\chi_{Q}\|_{p(\cdot)}\geq C_{1}|Q|^{(1/p(\cdot))_{Q}},
\end{equation}
for some constant $C_{1}>0.$

Now assume that $1/p(\cdot)\in BLO^{1/\log}$, then there exists $C_{2}$
such that
\begin{equation}
\label{estim_chi_norm_upper}
|Q|^{(1/p)_Q}=|Q|^{\frac{1}{|Q|}\int\limits_Q\frac{1}{p(x)}dx-\frac{1}{p_{+}(Q)}
+\frac{1}{p_{+}(Q)}}
\end{equation}
$$
\geq|Q|^{\frac{C}{\ln(e+1/|Q|)}+\frac{1}{p_{+}(Q)}}\geq
C_{2}\cdot |Q|^{\frac{1}{p_{+}(Q)}}.
$$
From (\ref{estim_chi_norm_lower}) and (\ref{estim_chi_norm_upper}) we obtain
\begin{equation}
\label{estim_chi_norm_uper_lover}
C_{3}\cdot|Q|^{1/p_+(Q)}\leq||\chi_Q||_{p(\cdot)}\leq C_4\cdot|Q|^{1/p_{+}(Q)}.
\end{equation}

\par Let $\mathcal{Q}=\{Q_1,Q_2,...\}$ denotes some partition of $[0;1]$.
Define on $[0;1]$ function $\tilde{p}(\cdot)$ in following way:
$\tilde{p}(x)=p_{+}(Q_i)$ when $x\in Q_i$.

\par Without restriction of
generality let consider case when $||f||_{p(\cdot)}=1.$  By
Proposition \ref{proposition_modular_in_var}  $\int_0^1|f(x)|^{p(x)}dx=1$ . Then we only need to
prove that
$$
\label{estim_G'_f=1}
\norm{\sum_i
\frac{||f\chi_{Q_i}||_{p(\cdot)}}{||\chi_{Q_i}||_{p(\cdot)}}\chi_{Q_i}(x)}_{p(\cdot)}\leq
C.
$$
\par By Proposition \ref{proposition_modular_in_var}
 we have
\begin{equation}
\label{estim_f_norm_like_modular}
||f\chi_{Q_i}||_{p(\cdot)}\leq\left(\int\limits_{Q_i}|f(x)|^{p(x)}dx\right)^{1/p_{+}(Q_i)}.
\end{equation}

Then by (\ref{estim_chi_norm_uper_lover}) and (\ref{estim_f_norm_like_modular})
$$
\int_{[0;1]} \left(\sum_i
\frac{||f\chi_{Q_i}||_{p(\cdot)}}{||\chi_{Q_i}||_{p(\cdot)}}\chi_{Q_i}(x)\right)^{\tilde{p}(x)}dx=
\sum_i
\int_{Q_i}\left(\frac{||f\chi_{Q_i}||_{p(\cdot)}}{||\chi_{Q_i}||_{p(\cdot)}}\right)^{p_{+}(Q_i)}\chi_{Q_i}(x)dx
$$
$$
=\sum_i |Q_i| \left(
\frac{||f\chi_{Q_i}||_{p(\cdot)}}{||\chi_{Q_i}||_{p(\cdot)}}\right)^{p_{+}(Q_i)}\leq\sum_i
|Q_i|\frac{\int\limits_{Q_i}|f(x)|^{p(x)}dx}{C_1|Q_i|}
$$
$$
=\frac{1}{C_1}\int_{[0;1]}|f(x)|^{p(x)}dx=\frac{1}{C_1}.
$$
Consequently we obtain
$$
 \norm{\sum_i
\frac{||f\chi_{Q_i}||_{p(\cdot)}}{||\chi_{Q_i}||_{p(\cdot)}}\chi_{Q_i}(x)}_{\tilde{p}(\cdot)}\leq
C.
$$
Using the fact that $p(x)\leq\tilde{p}(x),\,x\in[0;1]$ and proposition
\ref{proposition_embeding_Lp_Lq} we obtain desired result.
\end{proof}

\begin{proof}[Proof of theorem \ref{theorem_BLO+C+G''}]
The proof of theorem can be obtained from analogous arguments as in proof of theorem \ref{theorem_BLO_G'}. But we will obtain this proof from more general proposition.
\par Consider exponent $p(\cdot)$ such that $1/p(\cdot)\in BLO^{1/\log}$, then by theorem \ref{theorem_BMO_BLO_VLO} there exits constant $c$ such $p(\cdot)+c\in\mathcal{B}.$
Using theorem \ref{theorem_G'_and_condition_A_imply_G''} and theorem \ref{theorem_BLO_G'} we obtain desired result.
\end{proof}

\begin{proof}[Proof of theorem \ref{theorem_BLO_G'_notG''}]

Let us show that  the function
$$
f(x)=\left\{\begin{array}{rcl}
\ln\ln(1/x) & \textrm{if} & x\in (0,e^{-1}];\\
0 & \textrm{if} & x\in(e^{-1},1],
\end{array}
\right.
$$
belongs to $BLO^{1/\log}.$

 Let $(a;b)\subset[0;1].$ Without loss of generality assume that $0\leq a<b\leq
 e^{-1}.$  On  $(a;b]$ define  the function
$$
h(x)=\int\limits_a^x\ln\ln(1/t)dt-(x-a)\ln\ln(1/x)-\frac{2(x-a)}{\ln(1/(x-a))}.
$$
We have
$$
h'(x)=\frac{x-a}{x\ln(1/x)}-2\cdot
\frac{\ln(1/(x-a))+1}{(\ln(1/(x-a)))^2},\,\,a<x\leq b.
$$
 Since the function $x\ln(1/x)$ on $(0;1)$ is increasing
$$
(\ln(1/(x-a)))^2(x-a)-2x\ln(1/x)(\ln(1/(x-a))+1)
$$
$$
\leq\ln\frac{1}{x-a}\left((x-a)\ln\frac{1}{x-a}-2x\ln\frac{1}{x}\right)\leq
-x\cdot\ln\frac{1}{x}\cdot\ln\frac{1}{x-a}<0.
$$
 This means that function $h$ is decreasing. From monotonicity of $h$ and $h(a+)=0$ follows thats
$$
\int\limits_a^b\ln\ln(1/x)dx-(b-a)\ln\ln(1/b)-\frac{2(b-a)}{\ln(1/(b-a))}\leq0.
$$
By the last inequality we get
\begin{equation}
\label{estim_blo_koeff}
\frac{1}{b-a}\int\limits_a^b\ln\ln(1/x)dx-\ln\ln(1/b)\leq
\frac{4}{\ln(e+1/(b-a))},
\end{equation}
and consequently $f\in BLO^{1/\log}.$

 Note that function $f$ is a classical example
discontinuous functions from $BMO^{1/\log}$ (see \cite{Sp}).  From
the well-known observation that a Lipschitz function preserves mean
oscillations it follows that the function $\sin(f(x))$ provides  an
example of a discontinuous bounded function from $BMO^{1/\log}.$
Lerner \cite{Le} proved that if $p(x)=p_{0}+\mu\sin(f(x)),\:x\in[0;1]$ where $p_{0}>0$ and $\mu$
sufficiently close to $0,$ then Hardy-Littlewood maximal operator is
bounded on $L^{p(\cdot)}[0;1]$. It is unknown whether $p(\cdot)\in
BLO^{1/\log}.$  Bellow we will construct a bounded  function (some
sense analogous of $\sin(f(x))$) which belongs to $BLO^{1/\log}.$

 Let  $d_n=e^{-e^n},\,\,n\in\{0\}\cup\mathbb{N}$ and $c_0=2/e$, $c_{2n+1}=c_{2n}-(d_n-d_{n+1})$,  $c_{2n+2}=c_{2n+1}-(d_n-d_{n+1})$,  $n\in\{0\}\cup\mathbb{N}$. Let us show that
 the non-negative bounded function
$$
g(x)=\left\{\begin{array}{rcl}
\ln\ln\frac{1}{c_{2n}+c_{2n+2}-x-d_n}-n & \textrm{if} & x\in (c_{2n+2};\,c_{2n+1}], n\in\{0\}\cup\mathbb{N}\vspace{1ex};\\
\ln\ln\frac{1}{x-d_n}-n & \textrm{if} & x\in(c_{2n+1};\,c_{2n}], n\in\{0\}\cup\mathbb{N}\vspace{1ex};\\
0 & \textrm{if} & x\in(2/e,1].
\end{array}
\right.
$$
belongs to $BLO^{1/\log}$ i.e. for all $(a;b)\subset[0;1]$ we have
\begin{equation}
\label{estim_g_BLO_log}
\frac{1}{b-a}\int_{(a;b)}g(x)dx-\inf\limits_{x\in(a;b)}g(x)\leq\frac{C}{\ln(e+1/(b-a))}.
\end{equation}

\par Note that $g(c_{2n})=0$, $g(c_{2n+1})=1$, $n\in
\{0\}\cup\mathbb{N}$ and on each set $[c_{2n+1};c_{2n}]$ function $g$ is strictly monotonic
and continuous. 
\par Let $(a;b)\subset[0;1]$, without lose of generality suppose that $b\leq2/e$. Consider three cases:
\par Case 1.  At least one point $c_{2n}$ belongs to interval $(a;b)$, where $n\in\{0\}\cup\mathbb{N}$;
\par Case 2.  Interval $(a;b)$ contains only one point like $c_{2n+1}$, where $n\in\{0\}\cup\mathbb{N}$;
\par Case 3.  Interval $(a;b)$ does not contain point $c_n$ for any $n\in\{0\}\cup\mathbb{N}$.

\par Define $m_a=\sup\{k\::\:a\leq c_k\}$, $m_b=\min\{k\::\:c_k\leq b\}$. Note that if $a>0$ then $m_a=\max\{k\::\:a\leq c_k\}$ and $m_a=\infty$ if $a=0$.
\par Consider case 1. Suppose that $m_a<\infty$, define $m'_a=\max\{k\::\:a\leq c_k \:\wedge\:g(c_k)=0\}$ and $m'_b=\min\{k\::\:c_k\leq b\:\wedge\: g(c_k)=0\}$. It is clear that $c_{m_a} \leq c_{m'_a} \leq c_{m'_b}\leq c_{m_b}$. We have
\begin{equation}
\label{represent_case_1}
\frac{1}{b-a}\int_{(a;b)}g(x)dx-\inf_{x\in(a;b)}g(x)=\frac{1}{b-a}
\int_{(a;b)}g(x)dx
\end{equation}
$$
=\frac{1}{b-a}\left(\int_a^{c_{m'_a}}+\int_{c_{m'_a}}^{c_{m'_b}}+\int_{c_{m'_b}}^b\right)g(x)dx=A_1+A_2+A_3.
$$

\par Let $c_{m'_a}<c_{m'_b}$. Using the fact that $g(2c_{2k+1}-x)=g(x)$ when $x\in[c_{2k+2};c_{2k+1}]$ we get
$$
(b-a)A_2=\int\limits_{c_{m'_a}}^{c_{m'_b}}g(x)dx=\sum\limits_{k=m'_b/2}^{(m'_a-2)/2}\int\limits_{c_{2k+2}}^{c_{2k}}g(x)dx=\sum\limits_{k=m'_b/2}^{(m'_a-2)/2}\left(\int\limits_{c_{2k+2}}^{c_{2k+1}}+\int\limits_{c_{2k+1}}^{c_{2k}}\right)g(x)dx
$$
$$
=\sum\limits_{k=m'_b/2}^{(m'_a-2)/2}\:\:2\int\limits_{c_{2k+1}}^{c_{2k}}g(x)dx=2\sum\limits_{k=m'_b/2}^{(m'_a-2)/2}\:\:\int\limits_{c_{2k+1}}^{c_{2k}}\left(\ln\ln\frac{1}{x-d_k}-k\right)dx.
$$
Note that by $c_{2k}-d_k=d_k$ and $c_{2k+1}-d_k=d_{k+1}$ we have
$$
(b-a)A_2=2\sum\limits_{k=m'_b/2}^{(m'_a-2)/2}\:\:\int\limits_{d_{k+1}}^{d_{k}}\left(\ln\ln\frac{1}{t}-k\right)dt\leq 2\sum\limits_{k=m'_b/2}^{(m'_a-2)/2}\:\:\int\limits_{d_{k+1}}^{d_{k}}\left(\ln\ln\frac{1}{t}-\frac{m'_b}{2}\right)dt
$$
$$
=2\sum\limits_{k=m'_b/2}^{(m'_a-2)/2}\:\:\int\limits_{d_{k+1}}^{d_{k}}\left(\ln\ln\frac{1}{t}-\ln\ln\frac{1}{d_{m'_b/2}}\right)dt
=2\int\limits_{d_{m'_a/2}}^{d_{m'_b/2}}\left(\ln\ln\frac{1}{t}-\ln\ln\frac{1}{d_{m'_b/2}}\right)dt.
$$
Now by the following estimation $b-a>(b-a)/2\geq  d_{m'_b/2}-d_{m'_a/2}$ and by (\ref{estim_blo_koeff}) we have
\begin{equation}
\label{estim_A_2_1}
A_2\leq\frac{1}{d_{m'_b/2}-d_{m'_a/2}}\int\limits_{d_{m'_a/2}}^{d_{m'_b/2}}\left(\ln\ln\frac{1}{t}-\ln\ln\frac{1}{d_{m'_b/2}}\right)dt
\end{equation}

$$
\leq\frac{4}{\ln(e+1/(d_{m'_b/2}-d_{m'_a/2}))}\leq\frac{4}{\ln(e+1/(b-a))}.
$$
If $m_a=\infty$ then 
\begin{equation}
\label{estim_A_2_2}
A_2\leq\frac{1}{d_{m'_b/2}-0}\int\limits_{0}^{d_{m'_b/2}}\left(\ln\ln\frac{1}{t}-\ln\ln\frac{1}{d_{m'_b/2}}\right)dt\leq\frac{4}{\ln(e+1/(b-a))}.
\end{equation}
\par Consider $A_1$. Let $c_{m_a}=c_{m'_a}$. Since $c_{m'_a}-d_{m'_a/2}=d_{m'_a/2}$ and using (\ref{estim_blo_koeff}) we get
\begin{equation}
\label{estim_A_1_1}
A_1=\frac{1}{b-a}\int\limits_a^{c_{m'_a}}\left(\ln\ln\frac{1}{x-d_{m'_a/2}}-\frac{m'_a}{2}\right)dx
\end{equation}
$$
=\frac{1}{b-a}\int\limits_{a-d_{m'_a/2}}^{c_{m'_a}-d_{m'_a/2}}\left(\ln\ln\frac{1}{t}-\frac{m'_a}{2}\right)dx\leq\frac{4}{\ln(e+1/(c_{m'_a}-a))}\leq\frac{4}{\ln(e+1/(b-a))}.
$$
Let $c_{m_a}\neq c_{m'_a}$ then $m_a=m'_a+1$ and $g(c_{m_a})=1$. Since $c_{m'_a}-d_{m'_a/2}=d_{m'_a/2}$ we get
\begin{equation}
\label{estim_A_1_2}
A_1\leq\frac{2}{b-a}\int\limits_{c_{m_a}}^{c_{m'_a}}\left(\ln\ln\frac{1}{x-d_{m'_a/2}}-\frac{m'_a}{2}\right)dx=
\end{equation}
$$
=\frac{2}{b-a}\int\limits_{c_{m_a}-d_{m'_a/2}}^{c_{m'_a}-d_{m'_a/2}}\left(\ln\ln\frac{1}{t}-\ln\ln\frac{1}{d_{m'_a/2}}\right)dx\leq\frac{8}{\ln(e+1/(b-a))}.
$$

\par Consider $A_3$. Let $c_{m_b}=c_{m'_b}$. Since $c_{m_b-2}-d_{(m_b-2)/2}=d_{(m_b-2)/2}$ we get
\begin{equation}
\label{estim_A_3_1}
A_3=\frac{1}{b-a}\int\limits_{c_{m_b}}^b\left(\ln\ln\frac{1}{c_{m_b}+c_{m_b-2}-x-d_{(m_b-2)/2}}-\frac{m_b-2}{2}\right)dx=
\end{equation}
$$
=\frac{1}{b-a}\int\limits_{c_{m_b}+c_{m_b-2}-b-d_{(m_b-2)/2}}^{c_{m_b-2}-d_{(m_b-2)/2}}\left(\ln\ln\frac{1}{t}-\frac{m_b-2}{2}\right)dt\leq\frac{4}{\ln(e+1/(b-a))}.
$$
If $c_{m_b}\neq c_{m'_b}$ then $m_b=m'_b-1$ and $g(m_b)=1$ we have
\begin{equation}
\label{estim_A_3_2}
A_3\leq\frac{2}{b-a}\int\limits_{c_{m'_b}}^{c_{m_b}}\left(\ln\ln\frac{1}{c_{m'_b}+c_{m'_b-2}-x-d_{(m'_b-2)/2}}-\frac{m'_b-2}{2}\right)dx=
\end{equation}
$$
=\frac{2}{b-a}\int\limits_{c_{m'_b}+c_{m'_b-2}-c_{m_b}-d_{(m'_b-2)/2}}^{c_{m'_b-2}-d_{(m'_b-2)/2}}\left(\ln\ln\frac{1}{t}-\frac{m'_b-2}{2}\right)dt\leq\frac{8}{\ln(e+1/(b-a))}.
$$

\par In case of $m'_a=m'_b$ desired result can be obtained from estimations of $A_1$ and $A_3$.

\vspace{2 mm}

\par Case 2. It is clear that in this case $c_{m_a}=c_{m_b}=c_n$ where $n$ is odd. Note that restriction of function $g$ on the interval $(c_{n+1};c_{n-1})$ has symmetry about $x=c_n$ line, therefore without loss of generality we can assume that $g(a)\geq g(b)$ then
\begin{equation}
\label{estim_example_case_2}
\frac{1}{b-a}\int\limits_a^bg(x)dx-g(b)=\frac{1}{b-a}\int\limits_a^b(g(x)-g(b))dx\leq
\end{equation}
$$
\leq\frac{2}{b-a}\int\limits_{c_n}^b\left(\left(\ln\ln\frac{1}{x-d_{(n-1)/2}}-\frac{n-1}{2}\right)-g(b) \right)dx\leq
$$
$$
\leq\frac{2}{b-c_n}\int\limits_{c_n}^b\left(\ln\ln\frac{1}{x-d_{(n-1)/2}}-\ln\ln\frac{1}{b-d_{(n-1)/2}}\right)\leq\frac{4}{\ln(e+1/(b-a))}.
$$

\vspace{2 mm}

\par Case 3. In this case by (\ref{estim_blo_koeff}) we get desired estimation.
\par Finally by the estimates (\ref{represent_case_1})-(\ref{estim_example_case_2}) and (\ref{estim_blo_koeff}) we get (\ref{estim_g_BLO_log}).
\vspace{3 mm}

\par Now let construct exponent $p(\cdot)$ such that $1/p(\cdot)\in BLO^{1/\log}$
but $G''$ property fails.
\par We choose real numbers $a$ and $b$ such that $0<a<b<1$, $a+b<1$. Consider sets $A$ and $B$
$$
A=\set{x\::\:g(x)\leq a},\qquad
B=\set{x\::\:g(x)\geq b}.
$$
It is clear that this sets are union of intervals and let denote they by $\Delta_n^a$ and $\Delta_n^b$ i.e.
$$
A=\mathop{\bigcup}\limits_{n\geq1}\Delta_n^a,\qquad B=\mathop\bigcup\limits_{n\geq1}\Delta_n^b.
$$

\par Let now construct exponent $p$ in following way
$$
p(x)=\left\{\begin{array}{rcl}
1/a & \textrm{if} & x\in A \vspace{1ex};\\
1/b & \textrm{if} & x\in B\vspace{1ex};\\
1/g(x) & \textrm{if} &
x\in[0;1]\backslash(A\cup B).\\
\end{array}
\right.
$$

It is clear that $p(\cdot)$ is continuous except point 0, where it has discontinuity and $1/p(\cdot)\in BLO^{1/\log}$.
\par Let consider the set of right side endpoints of intervals from $A$. Let make partition of $[0;1]$ by these points. So we will get sequence of disjoint intervals $\Delta_n$ such that
$\Delta_n^a\cup \Delta_n^b\subset \Delta_n$.
\par Let $\delta_k=\min\{|\Delta_k^a|,|\Delta_k^b|\}$. Since $\delta_k\leq\min\{|\Delta_n^a|, |\Delta_n^b|\}$ for all $n\leq k$ then for each $n$, $n\leq k$ we can choose intervals $\Delta_n^{a'}\subset\Delta_n^a$ and $\Delta_n^{b'}\subset\Delta_n^b$ such that $\delta_k=|\Delta_n^{a'}|=|\Delta_n^{b'}|$.
\par Now for each $k$ we construct functions $f_k$ and $g_k$ in following way $f_k(x)=\chi_{\cup_{n\leq k}\Delta_n^{a'}}(x)$ and $g_k(x)=\chi_{\cup_{n\leq k}\Delta_n^{b'}}(x)$.
\par Let now check property $G$ of $L^{p(\cdot)}[0;1]$
$$
\sum_{n=1}^k||f_k\chi_{\Delta_n}||_{L^{1/a}}\cdot||g_k\chi_{\Delta_n}||_{L^{1/b}}=
\sum_{n=1}^k||\chi_{\Delta_n^{a'}}||_{L^{1/a}}\cdot||\chi_{\Delta_n^{b'}}||_{L^{1/b}}=
$$
$$
=\sum_{n=1}^k|\Delta_n^{a'}|^{a}\cdot|\Delta_n^{b'}|^{b}=k\cdot \delta_k^{a+b}.
$$
On the other hand
$$
||f_k||_{L^{1/a}}\cdot||g||_{L^{1/b}}=\left(\sum_{n=1}^k|\Delta_n^{a'}|\right)^{a}\cdot\left(\sum_{n=1}^k|\Delta_n^{b'}|\right)^{b}=\left(k\cdot \delta_k\right)^{a+b}.
$$
Property $G$ states that, there exits absolute constant $C$ such that
$$
k\cdot \delta_k^{a+b}\leq C\cdot \left(k\cdot \delta_k\right)^{a+b},
$$
we have
$$
k^{1-a-b}\leq C.
$$
The last estimation is impossible since $a+b<1$ and $k^{1-a-b}\to+\infty$, $k\to+\infty$.
\par Using theorem \ref{theorem_equval_G_l_estimates} and theorem \ref{theorem_G_imply_G'_G''} we conclude that $L^{p(\cdot)}[0;1]$ does not have property $G''$.
\par Note that $1/(p(\cdot)+c)\in BLO^{1/\log}$ for all $c>0$. Consequently exponents $p(\cdot)+c$ give us the spaces with same property.
\par Proof of the second part of theorem \ref{theorem_BLO_G'_notG''}. Note that by theorem \ref{theorem_BMO_BLO_VLO} and theorem \ref{theorem_G'_and_condition_A_imply_G''} we conclude that space $L^{(p(\cdot)+c)'}[0;1]$ possesses property $G''$ for some constant $c>0$.
It is clear that space $L^{(p(\cdot)+c)'}[0;1]$ does not have property $G'$ (because $L^{(p(\cdot)+c)}[0;1]$ does not have property $G''$).

\end{proof}


\end{document}